\documentclass[11 pt]{amsart}

\usepackage{graphicx} % Required for inserting images
\usepackage{amsfonts}
\usepackage{tikz-cd}
\usepackage{tikz}
\usepackage{tikz-3dplot}
\usepackage{tcolorbox}
\usepackage{anyfontsize}
\usepackage{tabularx}
\usepackage{scalerel}
\usepackage{bm}
\usepackage{tikz}
\usetikzlibrary{automata, arrows.meta, positioning}
\usetikzlibrary{shapes,shapes.geometric,arrows,fit,calc,positioning,automata,}

\usepackage{amssymb}

\usepackage{amsmath}
\usepackage{amsthm}
\usepackage{tabto}
\usepackage{mathtools}
\usepackage{marginnote}
\usepackage{MnSymbol}
\usepackage{enumerate}% http://ctan.org/pkg/enumerate

\newtheorem{theorem}{Theorem}[section] 
\newtheorem{theoremAlph}{Theorem}
\newtheorem*{theorem*}{Theorem}

\newtheorem{lemma}[theorem]{Lemma}
\newtheorem*{lemma*}{Lemma}
\newtheorem*{corollary*}{Corollary}
\newtheorem*{conjecture}{Conjecture}
\newtheorem*{fact}{Fact}
\newtheorem{proposition}[theorem]{Proposition}
\newtheorem*{proposition*}{Proposition}

\newenvironment{question}[1][Question]{\par\noindent\textbf{#1:} \rmfamily}{\medskip}

\theoremstyle{definition}
\newtheorem{definition}{Definition}[section]
\theoremstyle{remark}

\title{Thick embeddings into the Heisenberg group and coarse wirings into groups with polynomial growth}
\author{Or Kalifa}
\date{September 2024}

\usepackage{setspace}
\usepackage{times}

\usepackage{color}   %May be necessary if you want to color links
\usepackage{hyperref}
\hypersetup{
    colorlinks=False, %set true if you want colored links
    linktoc=all,     %set to all if you want both sections and subsections linked
}

\makeatletter
\def\@settitle{\begin{center}%
  \baselineskip14\p@\relax
  \normalfont\LARGE\bfseries
  \@title
  \end{center}%
}
\makeatother

\begin{document}

\newcommand{\R}{\mathbb{R}}
\newcommand{\N}{\mathbb{N}}
\newcommand{\Z}{\mathbb{Z}}
\newcommand{\G}{\Gamma}
\newcommand{\ra}{\rightarrow}
\newcommand{\g}{\gamma}
\newcommand{\ps}{\psi}
\newcommand{\vp}{\varphi}
\newcommand{\ve}{\varepsilon}
\newcommand{\e}{\mathsf{e}}

\begin{abstract}
We bound the volume of thick embeddings of finite graphs into the Heisenberg group, as well as the volume of coarse wirings of finite graphs into groups with polynomial growth. This work follows the work of Kolmogorov-Brazdin, Gromov-Guth and Barret-Hume on thick embeddings of graphs (or complexes) into various spaces. We present here a conjecture of Itai Benjamini that suggest that the lower bound of the volume of thick embeddings of finite graphs into locally finite, non-planar, transitive graphs, obtained by the separation profile, is tight. Let $Y$ be a Cayley graph of a group with polynomial growth, we prove that any finite bounded-degree graph $G$ admits a coarse $C\log(1+|G|)$-wiring into $Y$ with the optimal volume suggested by the conjecture.
Additionally, for the concrete case where $Y$ is a Cayley graph of the 3 dimensional discrete Heisenberg group, we prove that any finite bounded-degree graph $G$ admits a $1$-thick embedding into $Y$, with optimal volume up to factor $\log^2(1+|G|)$.
\end{abstract}
\maketitle
\tableofcontents{}

\section{Introduction} \label{sec: Introduction}
Kolmogorov and Barzdin were initially interested in the realization of finite graphs as physical networks in space, drawing inspiration from real-world examples like neurons in the brain or logical networks in computers \cite{KB}. To model physical networks, they required that the image of edges will have some thickness, akin to real wires that possess a positive thickness. This thickness constraint also provides a framework for examining the geometric efficiency of embeddings in \( \mathbb{R}^3 \). Without thickness, a graph can be embedded in a ball of arbitrarily small size, but with thickness, a meaningful lower bound exists.

\begin{definition}(Thick Embedding)
Let \( G \) be a graph and \( Y \) be a metric space. A continuous function \( f: G \to Y \) is called a \textbf{\( T \)-thick embedding} if \( d_Y(f(a), f(b)) \geq T \) whenever \( a \) and \( b \) are: two distinct vertices; a vertex and an edge that does not contain it; or two non-adjacent edges. If \( Y \) admits a measure, we define the \textbf{volume} of \( f \), \( \text{vol}(f) \), to be the measure of the 1-neighborhood of \( f(G) \subseteq Y \).
\end{definition}

Since the function \( f \) from \( G \) to \( Y \) is continuous, \( G \) must have a topological realization. As introduced in \cite{BH}, the \textbf{topological realization} of a graph \( G = (V, E) \) is the topological space obtained from a disjoint union of unit intervals indexed by \( e \in E \), with endpoints labeled by the two elements contained in \( e \). All endpoints labeled by the same vertex \( v \in V \) are identified. We will refer to both the graph and its topological realization using the same letter.

Additionally, since we aim to embed objects inside graphs, we need to define a \textbf{measure} on a graph \( Y \). Viewing \( Y \) as its topological realization, we define the sigma algebra \( \Sigma \) as \( \{A \subseteq Y \mid \forall e \in EY: A \cap e \text{ is Lebesgue measurable in } e \cong [0,1]\} \), and the measure \( \mu \) on \( \Sigma \) by \( \mu(A) = \sum_{e \in E} \lambda(A \cap e) \), where \( \lambda \) is the Lebesgue measure.
In particular, if \( A \) is a union of edges, then \( \mu(1\text{-neighborhood of }A) \) is the number of edges that touch \( A \). Therefore, it is no greater than \( \sum_{v \in (VY) \cap A} \text{deg}(v) \). In our case $Y$ will be Cayley graph of a finitely generated group, therefore for every $A\subseteq Y$ which is a union of edges we will have $\mu(A)\leq C\cdot |(VY)\cap A|$ for some constant $C$ (the size of the generating set).

\smallskip

Kolmogorov and Barzdin provided bounds on the volume of thick embeddings of finite graphs into \( \mathbb{R}^3 \) \cite{KB}. Later, Gromov and Guth extended these bounds to finite simplicial complexes embedded in Euclidean spaces \cite{GG}. Barrett and Hume further bounded the volume of thick embeddings of finite graphs into symmetric spaces \cite{BH}.

\smallskip
The separation profile, initially introduced in \cite{BST}, was later defined in \cite{H} as follows:
Given a finite graph \(G\) with \(n\) vertices, the cut size of \(G\), \(\text{cut}(G)\), is the minimal cardinality of a set of vertices \(S\) such that any connected component of \(G \setminus S\) has at most \(n/2\) vertices. The \textbf{separation profile} \(\text{sep}_X\) of a graph \(X\) evaluated at \(n\) is the maximum of the cut sizes of all subgraphs of \(X\) with at most \(n\) vertices.

Itai Benjamini explored the embedding of finite graphs into infinite transitive graphs. Drawing inspiration from the separation profile concept, and observing that this profile provides a lower bound for the embedding volume, he proposed the following conjecture:

\begin{conjecture}
Let \( d\geq 3 \), and let \( Y \) be a locally finite, non-planar, transitive graph with separation profile \( \simeq n^\alpha \) for some \( \alpha \in (0,1] \). There exists a constant \( C \) such that for every graph \( G \) with \( n \) vertices and maximum degree \( \leq d \), there is a $1$-thick embedding of \( G \) into \( Y \) with volume \( \leq Cn^{1/\alpha} \). Maybe up to a factor of $\text{polylog}(n)$.
\end{conjecture}

The symbol \( \simeq \) is derived from the definition of \( \lesssim \) for functions from \( \mathbb{N} \to \N\cup \{\infty\} \)\footnote{We include $\infty$ here, as this allows us to apply the notation to the function $wir_{X \to Y}^k$, which was introduced in \cite{BH} and may return the value $\infty$.
}. Two such functions, \( f \) and \( g \), satisfy \( f \lesssim  g \) if there exists a constant \( C \) such that \( f(n) \leq Cg(Cn)+C\) for every \( n \in \mathbb{N} \), we write \( f \simeq g \) if both \( f \lesssim g \) and \( f \gtrsim g \).

The conjecture suggest that the lower bound given by the separation profile is tight. The following result illustrates this lower bound:

\begin{fact}\label{fact}
Let \( d\geq 3 \), and let \( Y \) be a graph with separation profile \( \simeq n^\alpha \) for $\alpha\in (0,1]$. There exists a constant \( c > 0 \) such that for every \( n \in \mathbb{N} \), there exists a graph with \( n \) vertices and maximum degree \( \leq d \), for which any $1$-thick embedding into \( Y \) has volume \( \geq cn^{1/\alpha} \).
\end{fact}
This result follows directly from Theorem 1.17 in \cite{BH}, see lemma \ref{lemma: 1.17 imply fact} for details.

Thus, the conjecture proposes an upper bound that matches the known lower bound.

It is currently known that the conjecture holds for Cayley graphs of the groups \( \mathbb{Z}^d \) for \( d \geq 3 \in \mathbb{N} \). Indeed, the separation profile of \( \mathbb{Z}^d \) is \( \simeq n^{(d-1)/d} \) \cite[Corollary 3.3]{BST}\footnote{The separation profile does not depend on the choice of the generating set \cite[Page 4]{BST}.}. Moreover, it is known that there exists a constant \( C \) such that any bounded degree graph with \( n \) vertices can be embedded into \( \mathbb{R}^d \) via a $1$-thick embedding with volume at most \( Cn^{d/(d-1)} \) (Dimension 3 was proven by Kolmogorov and Barzdin in \cite[Page 195, Theorem 1]{KB}. Guth and Portnoy generalize it to all dimensions in \cite[Theorem 1.3]{BH}). This result can be adapted to give a $1$-thick embedding into a Cayley graph of \( \mathbb{Z}^d \); see the second appendix for details.

\smallskip 

Another important concept is the one of \textbf{coarse $\bm k$-wiring}, as introduced in \cite{BH}:
\begin{definition}
Let $\Gamma, \Gamma'$ be graphs. A \textit{\textbf{wiring}} of $\Gamma$ into $\Gamma'$ is a continuous map $f : \Gamma \to \Gamma'$ which maps vertices to vertices and edges to unions of edges. A wiring $f$ is a \textit{\textbf{coarse $k$-wiring}} if
\begin{enumerate}
    \item The restriction of $f$ to $V\Gamma$ is $\leq k$-to-$1$ i.e. $|\{ v \in V\Gamma | f(v) = w \}| \leq k$ for all $w \in V\Gamma'$.
    \item Each edge $e \in E\Gamma'$ is contained in at most $k$ of the image paths of $f$  i.e. $|\{ t \in E\Gamma | e\in f(t) \}| \leq k$ for each $e \in E\Gamma'$.
\end{enumerate}
\end{definition}
 The \textbf{\textit{volume} }of a wiring $vol(f)$ is the number of vertices in its image. The \textbf{diameter} of $f$ is the diameter of its image.

\smallskip
In this paper, we will prove two main theorems: Theorem \ref{thm:A} proves a weaker version of the conjecture for Cayley graphs of groups with polynomial growth:

\begin{theoremAlph}\label{thm:A}
Let \( d \in \mathbb{N} \) and let \( \Gamma \) be a group with polynomial growth of order \( \alpha > 1 \). Let \( S \subseteq \Gamma \) be a finite generating set. There exists a constant \( C = C(d, \Gamma, S) \) such that for any graph \( G \) with \( n \) vertices and maximum degree \( \leq d \), there is a coarse \( C\log(1+n) \)-wiring of $G$ into \( \text{Cay}(\Gamma, S) \) with diameter at most \( Cn^{\frac{1}{\alpha-1}} \) and volume at most \( Cn^{\frac{\alpha}{\alpha-1}} \).
\end{theoremAlph}

The volume obtained here aligns with the prediction of the conjecture, as a group with polynomial growth of order \( \alpha \) has growth \( \simeq n^\alpha \) (see section \ref{subsection: Cayley graph}), and therefore a separation profile \( \simeq n^{\frac{\alpha - 1}{\alpha}} \) (see \cite[Theorem 7]{HMT}\footnote{From the claims in the first appendix, we infer that the different definition of $\simeq$ here and in \cite{HMT} does not affect the result - we use here the fact that the separation profile and the growth of a group are monotone increasing functions.}). However, this result only proves a weaker version of the conjecture, as we prove the existence of a coarse \(C\log(1+n) \)-wiring rather than a $1$-thick embedding. 

The second result, Theorem \ref{thm:B}, successfully establishes a 1-thick embedding with near-optimal volume for bounded-degree finite graphs into the 3-dimensional discrete Heisenberg group, denoted by \( \mathbb{H} \):

\begin{theoremAlph}\label{thm:B}
Let \( d \in \mathbb{N} \) and let \( S \subseteq \mathbb{H} \) be a finite generating set. There exists a constant \( C \) such that for every graph \( G \) with \( n \) vertices and maximum degree \( \leq d \), there is a 1-thick embedding of \( G \) into \( \text{Cay}(\mathbb{H}, S) \) with diameter \( \leq Cn^{\frac{1}{3}}\log(1+n) \) and volume \( \leq Cn^{\frac{4}{3}}\log^2(1+n) \).
\end{theoremAlph}

The obtained bound is almost as predicted by the conjecture. Since the growth of \( \mathbb{H} \) is \( \simeq n^4 \), the separation profile is \( \simeq n^{\frac{3}{4}} \), leading the conjecture to expect a volume bound of \( \leq Cn^{\frac{4}{3}} \). We achieve this result, up to factor of \( \log ^2 (1+n) \).

\subsection{Combinatorial Embedding}
So far, we have defined two types of embeddings: thick embedding and coarse $k$-wiring. Now, we introduce two additional types of embeddings, which are almost identical to the previous ones but are more convenient to work with in our settings.

\begin{definition}[Combinatorial Wiring]
Let $G$ and $Y$ be graphs. A \textbf{combinatorial wiring} of $G$ into $Y$ is a function 
$$f: VG \cup EG \to \bigcup_{n \in \N} (VY)^n$$ 
such that the image of each vertex is a vertex, and the image of an edge $e=\{u, v\} \in EG$, denoted by $f_e$, is a walk with endpoints $fu, fv$. We call this walk a \textbf{road}.\\
\end{definition}
Let $f: G \to Y$ be a combinatorial wiring. The \textbf{load} of a vertex $y \in VY$ is defined as
$$
\text{load}_f(y) := \max \left( |\{e \in EG \mid y \in f_e\}|, |\{v \in VG \mid fv = y\}| \right).
$$
The \textbf{load} of $f$ is 
$$
\text{load}(f) := \sup_{y \in VY} \text{load}_f(y).
$$
A combinatorial wiring with load $\leq k$ is called a \textbf{combinatorial $\bm k$-wiring}.\\
The \textbf{image} of $f$ is the set of all vertices in $Y$ that are in $f(V)$ or on any road of $f$.\\
The \textbf{volume} and \textbf{diameter} of $f$ are the cardinality and diameter of $\text{Im}f$.\\

\begin{definition}[combinatorial embedding]
A combinatorial wiring $f:G\ra Y$ is called a \textbf{combinatorial embedding} if:
\begin{enumerate}
    \item $f|_{VG}$ is injective.
    \item "\textit{Roads do not intersect}" — for any two non-adjacent edges $e, t \in EG$, we have $f_e \cap f_t = \emptyset.$
    \item "\textit{Vertices are isolated from roads}" — for any edge $e \in EG$ and any vertex $v \notin e$, we have $fv \notin f_e.$
\end{enumerate}
\end{definition}

\smallskip
It is straightforward to show that for graph $G$ and graph $Y$ with maximal degree $\leq d$ :
\begin{itemize}
    \item If there is a combinatorial embedding of $G$ into $Y$, then there is a $1$-thick embedding of $G$ into $Y$ such that the diameter increases\footnote{Recall that the diameter of a combinatorial wiring is defined as the diameter of set of vertices, while the diameter of a thick embedding also considers points that are not vertices. This explains the difference of 2.
} by at most $2$ and the volume increases\footnote{Remember that the volume of a thick embedding is the measure of the $1$-neighborhood of the image, which is bounded from above by the maximal degree of the graph times the number of vertices in the image.} by at most a factor of $d$.
    \item If there is a combinatorial $k$-wiring of $G$ into $Y$, then there is a coarse $k$-wiring of $G$ into $Y$ with same volume and the diameter increases by at most $2$.
\end{itemize}
\subsection{Proof Overview}
\subsubsection{Theorem \ref{thm:A}}
Theorem \ref{thm:A} will be proven using a probabilistic method. Specifically, this will be done in the proof of the following Theorem:

\begin{theorem}[Random combinatorial wiring into group with polynomial growth]
\label{thm: Random combinatorial wiring}
Let \( d \in \mathbb{N} \) and let \( \Gamma \) be a group with polynomial growth of order \( \alpha > 1 \). Let \( S \subseteq \Gamma \) be a finite generating set. There exists a constant \( C =C(d, \G, S)\) such that for any graph \( G \) with \( n \) vertices and maximum degree \( \leq d \), there is a combinatorial $C\log(1+n)$-wiring of \( G\) into \(Cay(\Gamma, S) \) with diameter at most $C n^{\frac{1}{\alpha-1}}$ and volume at most $Cn^\frac{\alpha}{\alpha-1}$. Moreover, any two neighboring vertices are mapped to different vertices.
\end{theorem}
This theorem clearly implies theorem \ref{thm:A} since the existence of combinatorial $C\log(1+n)$-wiring implies the existence of coarse $C\log(1+n)$-wiring with same diameter and volume, up to a constant.

This theorem requires that the desired combinatorial wiring maps any two neighboring vertices to different vertices, a condition that is not needed for Theorem \ref{thm:A}. However, this property will be useful in the proof of Theorem \ref{thm:B}.

\subsubsection{Theorem \ref{thm:B}}
To prove Theorem \ref{thm:B}, we apply Theorem \ref{thm: Random combinatorial wiring} along with three additional propositions:

\begin{proposition}[Burden relief]
\label{propsition: Burden relief}
There is an absolute constant \( C \) such that:
For any finite graph \( G \) and a combinatorial \( k \)-wiring \( f: G \to  Cay(\mathbb{H},\{x,y\}) \)\footnote{The Cayley graph of \( \mathbb{H} \) with respect to the generating set \( \{x, y\} \) (see section \ref{sec: Heisenberg group})} where \( u \sim v \) implies \( fu \neq fv \), there exists a combinatorial embedding \( \tilde{f}: G \to Cay(\mathbb{H},\{x,y,z\}) \) that satisfies
$$ \text{diam}(\tilde{f}) \leq Ck \cdot \text{diam}(f) \quad \text{and} \quad \text{vol}(\tilde{f}) \leq Ck^2 \cdot \text{vol}(f). $$
\end{proposition}

\begin{proposition}[Removing $z$]
\label{proposition: removing z}
There is an absolute constant $C$ such that:
If $f:G\ra Cay(\mathbb{H},\{x,y,z\})$ is a combinatorial embedding, then there is a combinatorial embedding $\tilde f:G\ra  Cay(\mathbb{H},\{x,y\})$ with $$diam(\tilde{f})\leq C\cdot diam(f)\quad\text{and}\quad vol(\tilde{f})\leq C\cdot vol(f).$$
\end{proposition}

\begin{proposition}[General Generating Set]
\label{proposition: general generating set}
For any finite generating set $S \subseteq \mathbb{H}$, there exists a constant $C=C(S)$ such that if $f:G \ra  Cay(\mathbb{H},\{x,y\})$ is a combinatorial embedding, then there is a combinatorial embedding $\tilde f:G\ra Cay(\mathbb{H},S)$ such that
$$diam(\tilde{f})\leq C\cdot diam(f)\quad\text{and}\quad vol(\tilde{f})\leq C\cdot vol(f).$$
\end{proposition}

So, for a finite generating set \( S \subseteq \mathbb{H} \), we apply Theorem \ref{thm: Random combinatorial wiring} to obtain a constant \( C \), such that for any bounded-degree graph with \( n \) vertices, there exists a combinatorial \( C \log(1+n) \)-wiring into \( Cay(\mathbb{H},\{x,y\}) \) with diameter \( \leq C n^{\frac{1}{3}} \) and volume \( \leq C n^{\frac{4}{3}} \). By sequentially applying these three propositions, we obtain a constant \( C' \) such that for any bounded-degree graph with \( n \) vertices, there is a combinatorial embedding into \( Cay(\mathbb{H},S) \) with diameter \( \leq C' n^{\frac{1}{3}} \log(1+n) \) and volume \( \leq C' n^{\frac{4}{3}} \log^2(1+n) \). This embedding can, of course, be converted into a 1-thick embedding with the same diameter and volume, up to a constant factor.

\smallskip

\begin{question}
The proof of Theorem \ref{thm:B}, which provide an embedding into the Heisenberg group, relies on a probabilistic method. Is it possible to construct a non-random embedding explicitly? Additionally, can the $\log^2 (1+n)$ factor be removed to achieve an embedding with optimal volume?
\end{question}

\bigskip
\subsubsection{Paper overview}
The next section introduces some definitions used throughout the paper. Section~\ref{sec: Random} proves Theorem~\ref{thm: Random combinatorial wiring}. Section~\ref{sec: Heisenberg group} introduces the Heisenberg group and some basic properties. Finally, Section~\ref{sec: Embedding into Heisenberg} proves the three propositions discussed above.

\section{Definitions} \label{sec: Definitions}
\subsection{Notations}
\begin{enumerate}[$\bullet$]
    \item By \textbf{graph}, we refer to an undirected graph without multiple edges or loops. If $G$ is a graph, then $VG$ is its vertex set, and $EG$ is its edge set.
    \item Let $(X,d)$ be a metric space and $A\subseteq X$ be a subset. We define $N_\ve(A)$ as the closed $\ve$-neighborhood of $A$, meaning
    \[
    N_\ve(A) := \{x \in X \mid d(x,A) \leq \ve\}.
    \]Let $x \in X$ and $r > 0$. We denote by $B(x, r)$ the open ball of radius $r$ centered at $x$. Occasionally, we add a subscript to $B$ to specify the metric space.
    \item If a point $p$ can be seen as a vector in $\Z^n$, we denote by $p_i$ its projection onto the $i$-th axis, for any $1 \leq i \leq n$.
\end{enumerate}

\subsection{Cayley Graph}\label{subsection: Cayley graph}
Let $\G$ be a finitely generated group and let $S \subseteq \G$ be a finite generating set. We define \({S^{\ast}}\) as the set of all finite words above the alphabet $S \cup S^{-1}$, including the empty word. An element $w = s_1 \dots s_n \in S^\ast$ is called an $S$-word, and its length is denoted by $|w|$.\\
A word $w = s_1 \dots s_n \in S^\ast$ has an evaluation, $\bar{w}$, which is the element $s_1 \dots s_n \in \G$. We often use the same notation for a word and its evaluation; the meaning should be clear from context.\\
The graph $\bm{Cay(\G, S)}$, called the \textbf{Cayley graph of $\bm{\G}$ with respect to $\bm{S}$}, is the graph where $\G$ is the vertex set, and there is an edge between $g \in \G$ and $h \in \G$ if and only if $g^{-1}h \in S \cup S^{-1}$.  An edge of this graph say to be \textbf{labeled} $s$ for $s\in S\cup S^{-1}$ if it is edge of the form $\{g,gs\}$\footnote{Notice that this edge is also labeled $s^{-1}$.}.\\
This graph gives rise to a metric and a word length on $\G$, denoted $d_S$ and $\lVert\cdot \rVert_S$, respectively.\\
For each $g \in \G$, the \textbf{word length} $\lVert g \rVert_S$ is the graph distance from $g$ to $1 \in \G$. It can also be defined as
\[
\lVert g \rVert_S := \min_{w \in S^\ast, \bar{w} = g} |w|.
\]
For $g, h \in \G$, their distance $d_S(g, h)$ is $\lVert g^{-1} h \rVert_S.$\\

A group $\G$ has \textbf{polynomial growth} if there is finite generating set $S\subseteq \G$ and an integer $d\in \N$ such that $\exists C\forall n: |\overline{B_{S}}(1, n)|\leq Cn^d$. The results of Bass, Guivarc’h, and Gromov show that if $\G$ has polynomial growth, there is an integer $d$ such that $$\exists c,C>0\forall n: cn^d\leq |\overline{B_{S}}(1, n)|\leq Cn^d$$ for any finite generating set $S \subseteq \G$ [\cite{Gro}, \cite{Bas}, \cite{Gui}]. This $d$ is called the \textbf{growth rate} or \textbf{growth order} of the group $\G$.

\subsection{Walk on a Cayley Graph}
Let $Y$ be a graph. A \textbf{walk} in $Y$ is a finite sequence of vertices $P = (v_1, \dots, v_n)$ such that $v_i$ is adjacent to $v_{i+1}$ (denoted by $v_i\sim v_{i+1}$).\\
Now, assume $Y = Cay(\G, S)$ for some group $\G$ and a finite generating set $S$. For $g \in \G$ and $w = s_1 \dots s_n \in S^\ast$, we denote by $(g; w)$ the walk $(g, gs_1,gs_1s_2, \dots, gs_1 \dots s_n)$, which is a walk in $Y$ from $g$ to $g\bar{w}$. 
If $w_1,w_2\in S^{\ast}$ are two words we can write $(g;w_1,w_2)$ instead of $(g;w_1w_2)$.
\section{Random wiring into groups with polynomial growth} \label{sec: Random}
We now present the proof of Theorem \ref{thm: Random combinatorial wiring}:
\begin{theorem*}
Let \( d \in \mathbb{N} \) and let \( \Gamma \) be a group with polynomial growth of order \( \alpha > 1 \). Let \( S \subseteq \Gamma \) be a finite generating set. There exists a constant \( C =C(d, \G, S)\) such that for any graph \( G \) with \( n \) vertices and maximum degree \( \leq d \), there is a combinatorial $C\log(1+n)$-wiring of \( G\) into \(Cay(\Gamma, S) \) with diameter at most $C n^{\frac{1}{\alpha-1}}$ and volume at most $Cn^\frac{\alpha}{\alpha-1}$. Moreover, any two neighboring vertices are mapped to different vertices.
\end{theorem*}

Denote \( Y := \text{Cay}(\Gamma, S) \), and for each \( R \in \mathbb{N} \), let \( B_R := \overline{B_Y}(1, R) \) denote the closed ball of radius \( R \) in \( Y \). Since the growth of \( \Gamma \) is of order \( \alpha \), there exist positive constants \( k \) and \( K  \) (depending only on \( \Gamma \) and \( S \)) such that for every \( R \in \mathbb{N} \):
\[
kR^\alpha \leq |B_R| \leq KR^\alpha.
\]

\subsubsection{Random combinatorial wiring}
Choose for every group element $g\in \G$ a geodesic $\g_{g}$ from $1$ to $g$ in the graph $Y$. For two group elements $g,h\in \G$ denote by $g\g_h$ the walk $(g;s_1,...,s_n)$ where $\g_h=(1;s_1,...,s_n)$.

For radius $r\in \N$ and directed graph $G$ with $n$ vertices, the associated random combinatorial wiring, $\mathsf f$, is created by selecting the positions of the vertices independently and uniformly at random inside the ball $B_r$. The positions of the roads are then determined by their endpoints and the predefined geodesics of the form \( \gamma_g \). Namely $\mathsf f$ is a function that receive $\vec g=(g_1,...,g_n) \in B_r^n$ and return the combinatorial wiring $\mathsf f(\vec g)$ which maps vertex $i$ to the group element $g_i$ and an edge $e=(i,j)$ to the road $g_i\g_{g_i^{-1}g_j}$ which is a walk from $g_i$ to $g_j$.

\subsubsection{proof of Theorem \ref{thm: Random combinatorial wiring}}
\begin{proof}
Let $G$ a finite graph with $n$ vertices and maximal degree $\leq d$. We want to find good combinatorial wiring of $G$ into $Y$, good in terms of small diameter, volume and load. Define the directed graph $\vec G=(V,E)$ to be a graph obtained from $G$ by choosing some directions to all of the edges. 

Set \( r := \lceil n^{\frac{1}{\alpha-1}} \rceil \) and let $\mathsf f$  be the random combinatorial wiring associated with $r$ and $\vec G$, which is a combinatorial wiring of $G$ into $Y$.

Through probabilistic arguments, we will show that there exists some \( \vec g\in B_r^n \) such that \( \mathsf{f}(\vec g) \) have small diameter, volume and load.

We will show that there exists \( n_0 = n_0(\Gamma, S, d) \in \mathbb{N} \) such that if \( n \geq n_0 \), the following inequalities hold:
\begin{align*}
\mathbb{P}[E_1] &:= \mathbb{P}[\exists p \in B_r : \text{$\mathsf{f}$ sends at least $\log n$ vertices to } p] \leq 0.1,\\
\mathbb{P}[E_2] &:= \mathbb{P}[\exists p \in B_{2r} : \text{at least } \log n \text{ roads of } \mathsf{f} \text{ pass through } p] \leq 0.1,\\
\mathbb{P}[E_3] &:= \mathbb{P}[\exists \{u, v\} \in E : \mathsf{f}u = \mathsf{f}v] \leq 0.1.
\end{align*}
After establishing these bounds, we will complete the proof as follows:

If $n<  n_0 $, select a combinatorial wiring that is injective on the vertices, ensuring that the images of neighboring vertices are distinct. Since there are only constant number of graphs with less than $n_0$ vertices, the load, diameter and volume of this wiring is bounded by some constant.

If \( n \geq n_0 \), there is a positive probability that none of the events \( E_1 \), \( E_2 \), or \( E_3 \) occur. Therefore, there exists some \( \vec g\in B_r^n \) such that the combinatorial wiring \( f := \mathsf{f}(\vec g) \) is a wiring where none of these events happen. This \( f \) is the desired combinatorial wiring, as:
\begin{itemize}
    \item The load is less than \( \log n \) since events \( E_1 \) and \( E_2 \) do not occur.
    \item The image of the wiring lies entirely within \( B_{2r} \), so the diameter is \( \leq 4r \leq 8n^{\frac{1}{\alpha-1}} \).
    \item The volume is at most \( |B_{2r}| \leq K(2r)^\alpha \leq 4^\alpha K n^{\frac{\alpha}{\alpha-1}} \).
    \item For every edge \( \{u, v\} \in E \), we have \( fu \neq fv \) since \( E_3 \) does not occur.
\end{itemize}

To prove the existence of such $n_0$ we will prove below that $\mathbb{P}[E_i]\stackrel{n\ra \infty}{\longrightarrow}0$ for $i=1,2,3$.
\end{proof}

\begin{lemma}
$\mathbb{P}[E_1]\stackrel{n\ra \infty}{\longrightarrow}0$.
\end{lemma}
\begin{proof}
We aim to define a random variable \( \mathsf{Q} \) that indicates the maximum number of vertices in \( V \) that \( \mathsf{f} \) sends to the same point. For each \( p \in B_r \), define the random variable $\mathsf Q_p:= |\{v \in V \mid \mathsf{f}v = p\}|$. In other words, \( \mathsf{Q}_p \) counts the number of vertices that \( \mathsf{f} \) maps to \( p \). Now define $\mathsf{Q} := \max_{p \in B_r} \mathsf{Q}_p$.

For any \( m \in \mathbb{N} \):
\[
\mathbb{P}[\mathsf{Q} \geq m] = \mathbb{P}\left[\bigcup_{p \in B_r} (\mathsf{Q}_p \geq m)\right] \leq \sum_{p \in B_r} \mathbb{P}[\mathsf{Q}_p \geq m] \leq |B_r| \binom{n}{m} \frac{1}{|B_r|^m} \stackrel{{n \choose m} \leq \frac{n^m}{m!}}{\leq} \frac{n^m}{|B_r|^{m-1} m!}.
\]

Since \( |B_r| \geq k n^{\frac{\alpha}{\alpha - 1}} \geq kn \), we get:
\[
\mathbb{P}[\mathsf{Q} \geq m] \leq \frac{n}{k^{m-1} m!}.
\]

In total, for $n\geq 3$ we have:
\[
\mathbb{P}[E_1] = \mathbb{P}[\mathsf{Q} \geq \log n] \leq \mathbb{P}[\mathsf{Q} \geq \lfloor \log n \rfloor] \leq \frac{n}{k^{\lfloor \log n \rfloor - 1} \lfloor \log n \rfloor!}\stackrel{n\ra \infty}{\longrightarrow}0.
\]
\end{proof}

For an edge $e=(u,v)\in E$ define the random road $\mathsf e$ as $\mathsf e:=\mathsf f_e$, which is equivalent to just choosing two group elements $g,h$ independently and uniformly at random from $B_r$ and return the walk $g\g_{g^{-1}h}$. 

\begin{lemma}
$\mathbb{P}[E_2]\stackrel{n\ra \infty}{\longrightarrow}0$.
\end{lemma}
\begin{proof}
For a point \( p \in B_{2r} \) and an edge \( e = (u, v) \in E \), the probability that the rnadom road \( \mathsf e \) contains \( p \) is given by:
\[
\mathbb{P}[p \in \mathsf e] = \frac{|\{(g, g') \in B_r^2 \mid p \in g \gamma_{g^{-1} g'}\}|}{|\{(g, g') \in B_r^2 \}|}
\leq \frac{\sum_{h \in B_{2r}} |\{g \in B_r \mid p \in g \gamma_{h}\}|}{|\{(g, g') \in B_r^2 \}|}
\leq \frac{K(2r)^\alpha (2r)}{(kr^\alpha)^2}.
\]

This last inequality holds because:\\
Numerator: There are at most \( |B_{2r}| \leq K(2r)^\alpha \) ways to choose \( h \). For each choice of \( \mathfrak{h} \), we must choose \( g \) such that \( p \in g \gamma_{h} \), or equivalently, \( g \in p \gamma_{h}^{-1} \). Since \(|p\gamma_{h}^{-1}|= |\gamma_{h}| \leq 2r \), there are at most \( 2r \) choices for \( g \).\\
Denominator: We have $|B_r^2|\geq (kr^\alpha)^2$.

Thus, we have:
\[
\mathbb{P}[p \in \mathsf{e}] \leq \frac{2^{\alpha+1}K r^{\alpha+1}}{k^2 r^{2\alpha}} = c \frac{1}{r^{\alpha-1}} \leq \frac{c}{n},
\]
where \( c := \frac{2^{\alpha+1}K}{k^2} \).

This handles the case of a single edge. Now, we consider all edges simultaneously. Ideally, the random roads would be independent, but since some edges share common vertices they are not independent. We address this issue using the low degree of \( G \) and a coloring trick. By \textit{Vizing's theorem}, the edges of \( G \) can be colored with \( d+1 \) colors such that each subgraph \( G_i = (V, E_i) \), which induced by all the edges of color $i$, is a matching graph (a graph with maximum degree 1).

Fix \( i \in \{1, \dots, d+1\} \). Since \( G_i \) is a matching graph, every pair of edges in \( G_i \) are non-adjacent, and thus their corresponding random roads are independent.

For \( p \in B_{2r} \), define the random variable \( \mathsf{load}_p \) as \( |\{e \in E \mid p \in \mathsf{e}\}| \), which counts the number of roads passing through \( p \). Similarly, define \( \mathsf{load}_p^i := |\{e \in E_i \mid p \in \mathsf{e}\}| \) to count the number of \( i \)-colored roads passing through \( p \). Set \( m := \lfloor \frac{\log n}{d+1} \rfloor \), and we obtain:
\[
n^{\frac{\alpha}{\alpha-1}} \cdot \mathbb{P}[\mathsf{load}_p^i \geq m] \leq n^{\frac{\alpha}{\alpha-1}} \binom{dn}{m} \left(\frac{c}{n}\right)^m \leq \frac{n^{\frac{\alpha}{\alpha-1}} (dc)^m}{m!} \stackrel{n \to \infty}{\longrightarrow} 0.
\]

The first inequality uses the fact that the random roads \( \{\mathsf{e} \mid e \in E_i\} \) are independent, and there are no more than \( dn \) edges in \( E_i \).

This holds for a specific point \( p \in B_{2r} \), but we need to bound the probability for all points. The number of points is:
\[
|B_{2r}| \leq K (2r)^\alpha \leq 4^\alpha K n^{\frac{\alpha}{\alpha-1}}.
\]

Thus, we have:
\[
\mathbb{P}[\exists p \in B_{2r} : \mathsf{load}_p^i \geq m] = \mathbb{P}\left[\bigcup_{p \in B_{2r}} \mathsf{load}_p^i \geq m\right] \leq 4^\alpha K n^{\frac{\alpha}{\alpha-1}} \cdot \max_{p \in B_{2r}} \mathbb{P}[\mathsf{load}_p^i \geq m] \stackrel{n \to \infty}{\longrightarrow} 0.
\]

Finally, to bound the total load for all edges in \( G \), note that if a point \( p \in B_{2r} \) is contained in at least \( A \) roads, then there must be some \( i \) such that \( p \) is contained in at least \( \lfloor A/(d+1) \rfloor \) $i$-colored roads. Therefore:

\[
\mathbb{P}[E_2] = \mathbb{P}[\exists p \in B_{2r} : \mathsf{load}_p \geq \log n] \leq \mathbb{P}\left[\exists p \in B_{2r}, i \in \{1, \dots, d+1\} : \mathsf{load}_p^i \geq m\right]
\]
\[
= \mathbb{P}\left[\bigcup_{i \in \{1, \dots, d+1\}} \exists p \in B_{2r} : \mathsf{load}_p^i \geq m\right] \leq (d+1) \max_{i \in \{1, \dots, d+1\}} \mathbb{P}[\exists p \in B_{2r} : \mathsf{load}_p^i \geq m] \stackrel{n \to \infty}{\longrightarrow} 0.
\]
\end{proof}

\begin{lemma}
$\mathbb{P}[E_3]\stackrel{n\ra \infty}{\longrightarrow}0$.
\end{lemma}
\begin{proof}
For \( e = (u, v) \in E \), define the event \( b_e := \{\mathsf{f}u = \mathsf{f}v\} \). We aim to compute the probability that some \( b_e \) occurs:
\[
\mathbb{P}[\exists e \in E : b_e] = \mathbb{P}\left[\bigcup_{e \in E} b_e\right] \leq |E| \cdot \max_{e \in E} \mathbb{P}[b_e] \leq \frac{dn}{|B_r|}.
\]

This last inequality holds because \( |E| \leq dn \), and it's easy to see that \( \mathbb{P}[b_e] = \frac{1}{|B_r|} \) for all \( e \in E \).

Since \( |B_r| \geq k r^\alpha \geq k n^{\frac{\alpha}{\alpha-1}} \), we obtain:

\[
\mathbb{P}[E_3] = \mathbb{P}[\exists e \in E : b_e] \leq \frac{d}{k n^{\frac{1}{\alpha-1}}} \stackrel{n \to \infty}{\longrightarrow} 0.
\]

\end{proof}

\section{Heisenberg group} \label{sec: Heisenberg group}
In this paper, the Heisenberg group $\mathbb{H}$ refers to the group of $3 \times 3$ upper-triangular matrices of the form:
\[
\begin{pmatrix}
1 & a & c \\
0 & 1 & b \\
0 & 0 & 1
\end{pmatrix}, \quad a, b, c \in \mathbb{Z}.
\]
The group operation is standard matrix multiplication:
\[
\begin{pmatrix}
1 & a & c \\
0 & 1 & b \\
0 & 0 & 1
\end{pmatrix}
\begin{pmatrix}
1 & a' & c' \\
0 & 1 & b' \\
0 & 0 & 1
\end{pmatrix}
=
\begin{pmatrix}
1 & a + a' & c + c' + ab' \\
0 & 1 & b + b' \\
0 & 0 & 1
\end{pmatrix}.
\]
This group is generated by the following three matrices:
\[
x: = \begin{pmatrix}
1 & 1 & 0 \\
0 & 1 & 0 \\
0 & 0 & 1
\end{pmatrix}, \quad
y := \begin{pmatrix}
1 & 0 & 0 \\
0 & 1 & 1 \\
0 & 0 & 1
\end{pmatrix}, \quad
z := \begin{pmatrix}
1 & 0 & 1 \\
0 & 1 & 0 \\
0 & 0 & 1
\end{pmatrix}.
\]

For simplicity, we now identify $\mathbb{H}$ with $\mathbb{Z}^3$. Under this identification, the group multiplication becomes:
\[
(a, b, c) \cdot (a', b', c') = (a + a', b + b', c + c' + b'a).
\]

The three generators are represented as:
\[
x = (1, 0, 0), \quad y = (0, 1, 0), \quad z = (0, 0, 1).
\]
They satisfy the following multiplication rules:
\begin{align*}
(a, b, c) \cdot x &= (a + 1, b, c), \\
(a, b, c) \cdot y &= (a, b + 1, c + a), \\
(a, b, c) \cdot z &= (a, b, c + 1).
\end{align*}
\begin{figure}[h] % Use the figure environment to add caption and reference
    \centering
\begin{tikzpicture}[scale=1]

    % Axes
    \draw[->] (-1,-1,-1) -- (2,-1,-1) node[anchor=north east]{$x$};
    \draw[->] (-1,-1,-1) -- (-1,2,-1) node[anchor=north west]{$z$};
    \draw[->] (-1,-1,-1) -- (-1,-1,2) node[anchor=south]{$y$};

    % Draw grid points in 3D
    \foreach \a in {-1,0,1} {
        \foreach \b in {-1,0,1} {
            \foreach \c in {-1,0,1} {
                \filldraw[black] (\a,\b,\c) circle (1.5pt);
            }
        }
    }

    % Optional: Draw connecting lines for better visibility
        
        \draw[blue] (1,-1,-1) -- (1,1,0);
        %\draw[blue] (1,0,-1) -- (1,1,-0.5);
        \draw[blue] (1,-1,0) -- (1,1,1);
        %\draw[blue] (1,0,0) -- (1,1,0.5);

        \draw[blue] (0,-1,-1) -- (0,1,1);
        \draw[blue] (0,-1,0) -- (0,0,1);
        \draw[blue] (0,0,-1) -- (0,1,0);

        \draw[blue] (-1,-1,-1) -- (-1,-1,1);
        \draw[blue] (-1,0,-1) -- (-1,0,1);
        \draw[blue] (-1,1,-1) -- (-1,1,1);

    \foreach \a in {-1,0,1} {
        \foreach \c in {-1,0,1} {
            \draw[orange] (\a,-1,\c) -- (\a,1,\c);
        }
    }
    \foreach \b in {-1,0,1} {
        \foreach \c in {-1,0,1} {
            \draw[black, thin] (-1,\b,\c) -- (1,\b,\c);
        }
    }

\end{tikzpicture}
\caption{Small portion of the graph $Cay(\mathbb{H},\{x,y,z\})$. The black edges are labeled $x$, the blue edges are labeled $y$ and the orange edges are labeled $z$.}
\end{figure}
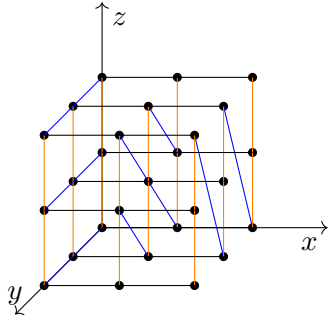
The set $\{x, y, z\}$ generates the group $\mathbb{H}$ because for any element \( g = (g_1, g_2, g_3) \in \mathbb{H} \), we have:
\[
g = y^{g_2} x^{g_1} z^{g_3}.
\]
Notice that the commutators are given by \( [x, z] = [y, z] = 1 \) and \( [x, y] = z \). Since \( z \) can be expressed in terms of \( x \) and \( y \), the set $\{x, y\}$ also generates $\mathbb{H}$.

In this paper, we will embed graphs into Cayley graphs of $\mathbb{H}$. Therefore, our notations will depend on the generating set we are considering, for a finite generating set $S\subseteq \mathbb{H}$ we denote $H_S:=Cay(\mathbb{H}, S)$. Since we primarily work with the generating sets $\{x, y\}$ and $\{x, y, z\}$, we introduce special notations for them. Let:

\[
H := H_{\{x, y\}}, \quad H_z := H_{\{x, y, z\}},
\]
\[
\lVert \cdot \rVert := \lVert \cdot \rVert_{\{x, y\}}, \quad \lVert \cdot \rVert_z := \lVert \cdot \rVert_{\{x, y, z\}},
\]
and
\[
d(\cdot, \cdot) := d_{\{x, y\}}(\cdot, \cdot), \quad d_z(\cdot, \cdot) := d_{\{x, y, z\}}(\cdot, \cdot).
\]
Now we examine some useful properties of the Heisenberg group $\mathbb{H}$:

\begin{lemma}
Let $S$ and $S'$ be two finite generating sets of $\mathbb{H}$. There exists a constant $m$ such that for every $g \in \mathbb{H}$, we have:
\[
\frac{1}{m} \lVert g \rVert_{S'} \leq \lVert g \rVert_{S} \leq m \lVert g \rVert_{S'}.
\]
\end{lemma}

\begin{proof}
This is a well-known result that holds for any finitely generated group.
\end{proof}

The following function will be important throughout this paper:

\begin{definition}[Expansion Function]
For $k \in \mathbb{N}$, define the function $\lambda_k: \mathbb{H} \to \mathbb{H}$ by:
\[
(a, b, c) \mapsto (ka, kb, k^2c).
\]
\end{definition}
The expansion function has several useful properties:

\begin{lemma}[Properties of the Expansion Function]
Let \( k \in \mathbb{N} \) and \( p \in \mathbb{H} \).
\begin{enumerate}
    \item \( \lambda_k \) is a homomorphism.
    \item \( \lVert \lambda_k(p) \rVert \leq k \lVert p \rVert \).
    \item If \( \lVert p \rVert = 1 \), then \( \lVert \lambda_k(p) \rVert = k \).
    \item If \( \lVert p \rVert \geq 2 \), then \( \lVert \lambda_k(p) \rVert \geq 2k \).
\end{enumerate}
\end{lemma}

\begin{proof}
$(1)$ is trivial, $(3)$ can be verified by checking all of the options. for $(2)$:
    Consider \( s_1, \dots, s_r \in \{x, x^{-1}, y, y^{-1}\} \) such that \( p = s_1 \dots s_r \) and \( r = \lVert p \rVert \). Then \( \lambda_k(p) = \lambda_k(s_1) \dots \lambda_k(s_r) = s_1^k \dots s_r^k \), and therefore:
    \[
    \lVert \lambda_k(p) \rVert \leq k r = k \lVert p \rVert.
    \]
    For $(4)$:
    If \( (|p_1|, |p_2|) \notin \{(0, 0), (0, 1), (1, 0)\} \), it is easy to see that \( \lVert \lambda_k(p) \rVert \geq 2k \). So, let us assume \( (|p_1|, |p_2|) \in \{(0, 0), (0, 1), (1, 0)\} \). Since \( p \notin \{1, x, y, x^{-1}, y^{-1}\} \) (as \( \lVert p \rVert \geq 2 \)), we infer that \( |p_3| \geq 1 \), which implies \( |(\lambda_k(p))_3| \geq k^2 \), and therefore \( \lVert \lambda_k(p) \rVert \geq 2k \).

Since, let:
\[
\lambda_k(p) = x^{a_1} y^{b_1} \dots x^{a_r} y^{b_r}
\]
be the shortest way to express \( \lambda_k(p) \) using the generators \( \{x, y\} \). Define:
\[
\alpha = |a_1| + \dots + |a_r|, \quad \beta = |b_1| + \dots + |b_r|.
\]
We have \( \lVert \lambda_k(p) \rVert = \alpha + \beta \). Consider the expression for \( (\lambda_k(p))_3 \):
\[
(\lambda_k(p))_3 = b_1 a_1 + b_2 (a_1 + a_2) + \dots + b_r (a_1 + \dots + a_r),
\]
which gives:
\[
k^2 \leq |b_1||a_1| + |b_2|(|a_1| + |a_2|) + \dots + |b_r|(|a_1| + \dots + |a_r|) \leq \alpha \beta.
\]

Thus, since \( 0 \leq (\alpha - \beta)^2 \), we have:
\[
\alpha^2 + \beta^2 \geq 2\alpha \beta.
\]
This leads to:
\[
(\alpha + \beta)^2 = \alpha^2 + \beta^2 + 2\alpha \beta \geq 4\alpha \beta \geq 4k^2,
\]
which implies:
\[
\lVert \lambda_k(p) \rVert = \alpha + \beta \geq 2k.
\]

\end{proof}
\section{Coarse to fine in the Heisenberg group} \label{sec: Embedding into Heisenberg}
Here we prove propositions \ref{propsition: Burden relief}, \ref{proposition: removing z}, \ref{proposition: general generating set}.

\subsection{Proposition \ref{propsition: Burden relief} - Burden relief}
We present here the proof of proposition \ref{propsition: Burden relief}, which transforms combinatorial $k$-wiring into a combinatorial embedding, increasing the diameter by a factor of $k$ and the volume by a factor of $k^2$. 

The ability to convert a combinatorial wiring into combinatorial embedding in the Heisenberg group is due to the presence of $z$ as a non-identity element in the center of $\mathbb{H}$. This allows us to play with the $z$ generators freely, as they commute with all other generators. Smart play with $z$ generators is the key for separating intersecting roads. For this reason, proposition \ref{propsition: Burden relief} provides a combinatorial embedding specifically into $H_z$ rather than into $H$.

Here is the proposition:
\begin{proposition*}[Burden relief]
There is an absolute constant \( C \) such that:
For any finite graph \( G \) and a combinatorial \( k \)-wiring \( f: G \to H \) where \( u \sim v \) implies \( fu \neq fv \), there exists a combinatorial embedding \( \tilde{f}: G \to H_z \) that satisfies
$$ \text{diam}(\tilde{f}) \leq Ck \cdot \text{diam}(f) \quad \text{and} \quad \text{vol}(\tilde{f}) \leq Ck^2 \cdot \text{vol}(f). $$
\end{proposition*}
\begin{proof}
Let $G$ be a finite graph and $f:G\ra H$ be a combinatorial $k$-wiring such that $u\sim v\Rightarrow fu\ne fv$.

The idea of the proof is to scale \( f \) by a large factor, providing enough space to resolve local intersections of roads. This resolution is achieved using colors. We assign a color to each vertex \( v \in V \), and instead of coloring each edge \( e \in E \), we assign colors to each pair \( (p, e) \), where \( p \in f_e \) and \( e \in E \). Using the colors as a guide, we manipulate the \( z \) generators to separate the roads.

We break the proof of the proposition into four steps:
 \begin{itemize}
     \item Assigning colors.
     \item Using the colors to define the new combinatorial wiring $\tilde f$. 
     \item Proving that $\tilde f$ is a combinatoral embedding.
     \item Proving that $\tilde f$ has small diameter and volume.
 \end{itemize}

\subsubsection{Assigning colors}
We aim do define coloring function $$c:V\cup \{(p,e)\mid p\in f_e\}\ra \{0,...,32k-1\}$$ with the following properties:
\begin{enumerate}
\item $c$ separates vertices - if $v\ne u\in V$ and $fu=fv$ then $c(v)\ne c(u)$.
\item $c$ separate roads - for non-adjacent edges $e,t\in E$ and points $p\in f_e,q\in f_t$ we have $$d(p,q)\leq 3\Rightarrow c(p,e)\ne c(p,t).$$
\item  $c$ separates vertices from inner points of roads - $$c(\{(p,e)\mid p\in f_{e},\text{$p$ is not an endpoint of $f_e$}\})\cap c(V)=\emptyset.$$
\end{enumerate}

For numbers $a,b\in \N$ we denote the remainder of \( a \) when divided by \( b \) by \( \operatorname{mod}_b(a) \in \{0,...,b-1\}\). 

Since \( f \) is a combinatorial $k$-wiring, there are functions \( \eta_V: V \to \{0, \dots, k-1\} \) and \( \eta_E: \{(p, e) \mid p \in f_e\} \to \{0, \dots, k-1\} \), which we denote collectively by \(\eta\), such that:
\begin{itemize}
    \item If \( u \neq v \in V \) and \( fu = fv \), then \( \eta(u) \neq \eta(v) \),
    \item and if \( e \neq t \in E \) and \( p \in f_e \cap f_t \), then \( \eta(p, e) \neq \eta(p, t) \).
\end{itemize}

For a point \( p  \in \mathbb{H} \), we define \(
\alpha(p) := 4  \operatorname{mod}_4(p_1) + \operatorname{mod}_4(p_2).
\) Thus, \( \alpha \) takes values in \( \{0, \dots, 15\} \), and \( \alpha(p) = \alpha(q) \) if and only if \( p_1 \equiv q_1 \pmod{4} \) and \( p_2 \equiv q_2 \pmod{4} \).

Now we are ready to define $c$:
For vertex $v\in V$, we define its color from $\{0,...,16k-1\}$ by $$c(v):=\alpha(fv)k + \eta(v).$$And for pair $(p,e)$ with $p\in f_e$ define its color from $\{0,...,32k-1\}$ by: 
\[
c(p, e) := \begin{cases}
    c(v) & \text{if } p = fv \text{ for some vertex } v \in e, \\
    (16 + \alpha(p))k + \eta(p,e) & \text{otherwise}.
\end{cases}
\]
This function is well-defined since \( u\sim v\Rightarrow fu \neq fv \).

It's easy to see that $c$ satisfy properties $(1)$ and $(3)$, to see it also satisfies property $(2)$:
Let \( e \) and \( t \) be two non-adjacent edges in \( E \), and take points \( p \in f_e \) and \( q \in f_t \).
Assume \( d(p, q) \leq 3 \) and suppose, for contradiction, that \( c(p, e) = c(q, t) \). This implies \( \alpha(p) = \alpha(q) \), and since \( d(p, q) \leq 3 \), we have \( p = q \). Now, there are two possible cases:
\begin{enumerate}
\item \( c(p, e) = c(u) \) for some vertex \( u \in e \). Since $c(p,e)<16k$ we must have \( c(q, t) = c(v) \) for some vertex \( v \in t \). So, $u$ and $v$ are different vertices with same $f$-image, hence their color is different which leads to $c(p,e)\ne c(q,t)$ - a contradiction!
\item \( c(p, e) = (16 + \alpha(p))k + \eta(p,e) \). Since $c(p,e)\geq 16k$ we have \( c(q, t) = (16 + \alpha(q))k + \eta(q,t) \). But since $p=q$ then $\eta(p,e)\ne \eta(q,t)$ which means $c(p,e)\ne c(q,t)$ - a contradiction!
\end{enumerate}

\subsubsection{Defining the combinatorial wiring $\tilde f$}
Let \( K := 1000k \), \( \lambda := \lambda_K \) and \( S := \{x, y, x^{-1}, y^{-1}\} \).

For vertex \( v \in V \), define:
\[
\tilde{f}v := z^{c(v)} \lambda(fv).
\]
Notice that \( \tilde{f}|_V \) is injective.

Next, we define the roads of \( \tilde{f} \):
For \( s \in S \) and \( \alpha, \alpha' \in \{0, \dots, 32k-1\} \), define the word:
\[
w_{s,\alpha,\alpha'} := \begin{cases} 
s^{500k + \alpha} y z^{\alpha' - \alpha} y^{-1} s^{500k - \alpha} & \text{if } s \in \{x, x^{-1}\}, \\
s^{500k + \alpha} x z^{\alpha' - \alpha} x^{-1} s^{500k - \alpha} & \text{if } s \in \{y, y^{-1}\}.
\end{cases}
\]
\begin{itemize}
    \item For any point \( p \in \mathbb{H} \), the walk \( (z^\alpha \lambda(p); w_{s, \alpha, \alpha'}) \) is a walk from \( z^\alpha \lambda(p) \) to \( z^{\alpha'} \lambda(ps) \).
    \item The length of this word is $K+2+|\alpha'-\alpha|\leq 1100k$.
    \item Every point on this walk is at most \( 600k \) \( \{x, y\} \)-distance away from one of its endpoints, which can easily be verified by replacing each letter \( z \) with the word \( xyx^{-1}y^{-1} \).
\end{itemize}

So For edge $e\in E$, if $f_e=(p_1,...,p_r)$ we let $\alpha_i:=c(p_i,e)$ for $i=1,...,r$ and we define 
\[
\tilde{f}_e := \left( z^{\alpha_1} \lambda(p_1); w_{p_1^{-1}p_{2}, \alpha_1, \alpha_2}, \dots, w_{p_{r-1}^{-1}p_{r}, \alpha_{r-1}, \alpha_r} \right).
\]
This walk is a concatenations of walks: from $z^{\alpha_1}\lambda(p_1)$ to $z^{\alpha_2}\lambda(p_2)$, from $z^{\alpha_2}\lambda(p_2)$ to $z^{\alpha_3}\lambda(p_3)$, and so on, until the walk from $z^{\alpha_{r-1}}\lambda(p_{r-1})$ to $z^{\alpha_r}\lambda(p_r)$. 

In particular it begins at \( z^{\alpha_1} \lambda(p_1)\) and ends at \( z^{\alpha_r} \lambda(p
_r)\) - which confirms that $\tilde f$ is combinatorial wiring, meaning the endpoints of each road are the images of the endpoints of the associated edge. 

\begin{tikzpicture}[remember picture]
    % First plot as a node
    \node (plot1) at (0,0) {
        \begin{tikzpicture}
            \tdplotsetmaincoords{70}{60}
            \begin{scope}[tdplot_main_coords,
                cube/.style={very thick,black},
                grid/.style={very thin,gray},
                axis/.style={->,blue,thin}]

            \draw[axis] (0,0,0) -- (4,0,0) node[anchor=west]{$x$};
            \draw[axis] (0,0,0) -- (0,4,0) node[anchor=west]{$y$};
            \draw[axis] (0,0,0) -- (0,0,3) node[anchor=west]{$z$};
            
            \draw[cube] (0,0,1) -- (2,0,1) -- (2,0.5,1.5) -- (2,0.5,2) --(2,0,1.5)--(4,0,1.5);
                
            \end{scope}
        \end{tikzpicture}
    };
    \node at (plot1.south) [above=130pt] {The walk $(z^2\lambda(p);w_{x,2,3})$.}; % Title for the first plot

    % Second plot as a node
    \node (plot2) at (8,0) 
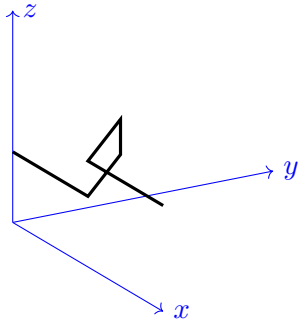
{ % Adjust horizontal positioning as needed
        \begin{tikzpicture}
            \tdplotsetmaincoords{70}{60}
            \begin{scope}[tdplot_main_coords,
                cube/.style={very thick,black},
                grid/.style={very thin,gray},
                axis/.style={->,blue,thin}]

            \draw[axis] (0,0,0) -- (4,0,0) node[anchor=west]{$x$};
            \draw[axis] (0,0,0) -- (0,4,0) node[anchor=west]{$y$};
            \draw[axis] (0,0,0) -- (0,0,3) node[anchor=west]{$z$};
            
            \draw[cube] (1,0,1) -- (1,2,2) -- (1.5,2,2) -- (1.5,2,1.5) --(1,2,1.5)--(1,4,2.5);
            
            \end{scope}
        \end{tikzpicture}
    };
    \node at (plot2.south) [above=130pt] {The walk $(z^3\lambda(p);w_{y,3,2})$.}; % Title for the first plot
    
    \end{tikzpicture}

\subsubsection{Proving that $\tilde f$ is a combinatorial embedding}\label{tilde f is embedding}
We already saw that $\tilde f|_V$ is injective so it remains to show that "Roads do not intersect" and "Vertices are isolated from roads".

\textbf{\textit{"Roads do not intersect"}}:
Let \( e,t\in E \) be non-adjacent edges, take \( \xi \in \tilde{f}_e \) and \( \xi' \in \tilde{f}_t \). By the definition of \( \tilde{f} \), there are vertices \( p\sim p', q\sim q' \) in the graph $H$ such that if we write \( \alpha := c(p, e), \alpha' := c(p', e) \) and \( \beta := c(q, t), \beta' := c(q', t) \) we get
\[
\xi \in \left( z^{\alpha} \lambda(p); w_{p^{-1}p', \alpha, \alpha'} \right) \quad \text{and} \quad \xi' \in \left( z^{\beta} \lambda(q); w_{q^{-1}q', \beta, \beta'} \right).
\]
We prove that \( \xi \neq \xi' \):

If \( d(\{p, p'\}, \{q, q'\}) \geq 2 \), then \( d(\{\lambda(p), \lambda(p')\}, \{\lambda(q), \lambda(q')\}) \geq 2K \). The point \( \xi \) is at most \( 600k \) \( \{x, y\} \)-distance away from the endpoints of the walk \( \left( z^{\alpha} \lambda(p); w_{p^{-1}p', \alpha, \alpha'} \right) \), and these endpoints are within \( 200k \) $\{x,y\}$-distance from \( \{\lambda(p), \lambda(p')\} \). Thus, \( d(\xi, \{\lambda(p), \lambda(p')\}) \leq 800k \). Similarly, \( d(\xi', \{\lambda(q), \lambda(q')\}) \leq 800k \), so \( \xi \neq \xi' \).

If \( d(\{p, p'\}, \{q, q'\}) \leq 1 \), then all distances between \( p,p', q, q' \) are at most 3. By property $(2)$ of \( c \), we get \( \{\alpha, \alpha'\} \cap \{\beta,\beta'\} = \emptyset \). If \( \xi = \xi' \), then, without loss of generality, \( \operatorname{mod}_K(\xi_3) \notin  \{\alpha, \alpha'\} \). Examining the definition of\( \left( z^{\alpha} \lambda(p); w_{p^{-1}p', \alpha, \alpha'} \right) \), we conclude that this is only possible if \( \{\operatorname{mod}_K(\xi_1), \operatorname{mod}_K(\xi_2)\} \in \{\{1, 500k + \alpha\}, \{1, 500k - \alpha\}\} \). However, this is impossible, since $\xi\in \left( z^{\beta} \lambda(q); w_{q^{-1}q', \beta, \beta'} \right)$ implies that the set $\{\operatorname{mod}_K(\xi_1), \operatorname{mod}_K(\xi_2)\}$ either contain 0 or is one of \( \{1, 500k + \beta\}, \{1, 500k - \beta\}\).

\smallskip

\textbf{\textit{"Vertices are isolated from roads"}}:
Let \( e \in E \) be an edge and \( v \notin e \) be a vertex, let $\xi \in \tilde f_e$, assume towards contradiction that $\tilde fv= \xi$. Notice that $(\operatorname{mod}_K((\tilde fv)_1), \operatorname{mod}_K((\tilde fv)_2))=(0,0)$, the only points in $\tilde f$ with this property are of the form $z^{c(p,e)}\lambda(p)$ for $p\in f_e$, hence $\xi = z^{c(p,e)}\lambda(p)$ for some $p\in f_e$. Recall that $\tilde fv=z^{c(v)}\lambda(fv)$.
\begin{itemize}
    \item If $p$ is inner point of $f_e$ then $c(v)\ne c(p,e)$ by property $(3)$ of $c$, which leads to $\tilde fv\ne \xi$.
    \item If $p$ is an endpoint of $f_e$ then $\xi=\tilde fu$ for some vertex $u\in e$, which implies $\tilde fv\ne \xi$ since $\tilde f|_V$ is injective. 
\end{itemize}

\subsubsection{Proving that $\tilde f$ has small diameter and volume}
For the diameter, observe that \( \text{Im}(\tilde{f}) \subseteq N_{800k}(\lambda_K(\text{Im}(f))) \) with respect to the $\{x, y\}$-metric, as demonstrated in \ref{tilde f is embedding}. Thus, the diameter of \( \tilde{f} \) in the $\{x, y\}$-metric is at most \( K \cdot \text{diam}(f) + 1600k \leq 2600k \cdot \text{diam}(f) \). Therefore, the diameter of \( \tilde{f} \), which is calculated with respect to the $\{x, y, z\}$-metric, is bounded from above by a constant multiple of the this bound.

Analyzing the volume is a bit more complicated: Notice that since \( f \) is a combinatorial $k$-wiring, we have \( \text{vol}(f)\geq \frac{|V| + \sum_{e \in E} |f_e|}{2k} \).
Now, observe that for each $e \in E$, the road \( \tilde{f}_e \) is a walk with at most $1100k$ times more edges than \( f_e \). Hence \[|\tilde{f}_e| \leq |\text{edges in } \tilde{f}_e| + 1 \leq 1100k |\text{edges in } f_e| + 1 \leq 2200k|f_e| + 1 \leq 2201k|f_e| .\]
\begin{itemize}
    \item The first inequality follows because the number of vertices in a graph is always at most the number of edges plus 1.
    \item The third inequality follows because the number of edges in a graph is at most half the maximal degree of the graph multiplied by the number of vertices, and since \( f_e \) is a subgraph of \( H \), its maximal degree is at most 4.
\end{itemize}
Thus, we obtain \[ \text{vol}(\tilde{f}) \leq |V| + \sum_{e \in E} |\tilde{f}_e| \leq 2201k(|V| + \sum_{e \in E} |f_e|) \leq 4402k^2 \text{vol}(f) .\]
\end{proof}

\subsection{Proposition \ref{proposition: removing z} - Removing $z$}
In proposition \ref{propsition: Burden relief}, we obtained a combinatorial embedding into $H_z$. However, as shown in Lemma \ref{proposition: general generating set}, in order to achieve a combinatorial embedding into $H_S$ for a general finite generating set $S$, we must begin with a combinatorial embedding into $H$. Therefore, proposition \ref{proposition: removing z} converts a combinatorial embedding into $H_z$ to a combinatorial embedding into $H$.
\begin{proposition*}[Removing $z$]
\label{lemma removing z}
There is an absolute constant $C$ such that:
If $f:G\ra H_z$ is a combinatorial embedding, then there is a combinatorial embedding $\tilde f:G\ra  H$ with $$diam(\tilde{f})\leq C\cdot diam(f)\quad\text{and}\quad vol(\tilde{f})\leq C\cdot vol(f).$$
\end{proposition*}
\begin{proof}
We have a combinatorial embedding \( f: G \to H_z \). Our goal is to expand \( f \) sufficiently so that we can replace each edge labeled \( z \) by the 4-edge walk \( xyx^{-1}y^{-1} \) while ensuring that the roads of the embedding remain disjoint.

Let \( \lambda := \lambda_2 \) and \( S := \{x, x^{-1}, y, y^{-1}, z, z^{-1}\} \). We want to define a new combinatorial embedding \( \tilde f: G\ra H \). We begin by defining \( \tilde{f}|_V := \lambda \circ f|_V \), which is injective because both \( f|_V \) and \( \lambda \) are injective.

For each generator \( s \in S \), we define the word:
\[
\tilde{s} := \begin{cases}
s^2 & \text{if } s \in \{x, y, x^{-1}, y^{-1}\}, \\
(xyx^{-1}y^{-1})^4 & \text{if } s = z, \\
(yxy^{-1}x^{-1})^4 & \text{if } s = z^{-1}.
\end{cases}
\]

Observe that for any point \( p \in \mathbb{H} \), the walk \( (\lambda(p); \tilde{s}) \) goes from \( \lambda(p) \) to \( \lambda(ps) \).

For an edge \( e = \{u, v\} \in E \), write \( f_e = (p; s_1, \dots, s_r) \) and define:
\[
\tilde{f}_e := (\lambda(p); \tilde{s_1}, \dots, \tilde{s_r}).
\]
This defines a walk with endpoints \( \tilde{f}u, \tilde{f}v \).

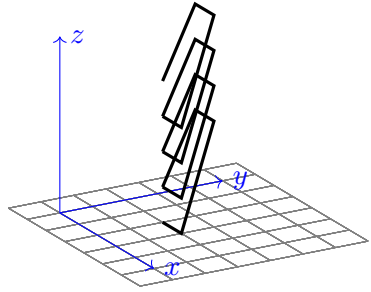
\begin{tikzpicture}[remember picture]
    % First plot as a node
    \node (plot1) at (0,0) {
        \begin{tikzpicture}
            \tdplotsetmaincoords{70}{60}
            \begin{scope}[tdplot_main_coords,
                cube/.style={very thick,black},
                grid/.style={very thin,gray},
                axis/.style={->,blue,thin}]
                
                % Draw grid, axes, and cube...
                \foreach \x in {-0.5,0,...,3}
                    \foreach \y in {-0.5,0,...,3} {
                        \draw[grid] (\x,-0.5) -- (\x,3);
                        \draw[grid] (-0.5,\y) -- (3,\y);
                    }
                \draw[axis] (0,0,0) -- (2.5,0,0) node[anchor=west]{$x$};
                \draw[axis] (0,0,0) -- (0,2.5,0) node[anchor=west]{$y$};
                \draw[axis] (0,0,0) -- (0,0,2.5) node[anchor=west]{$z$};

            % drow the walk
            \draw[cube] (0.5,0.5,0) -- (0.5,0.5,0.5);
                
            \end{scope}
        \end{tikzpicture}
    };
    \node at (plot1.south) [above=120pt] {The walk $(g;z)$.}; % Title for the first plot

    % Second plot as a node
    \node (plot2) at (8,0) { % Adjust horizontal positioning as needed
        \begin{tikzpicture}
            \tdplotsetmaincoords{70}{60}
            \begin{scope}[tdplot_main_coords,
                cube/.style={very thick,black},
                grid/.style={very thin,gray},
                axis/.style={->,blue,thin}]

                % Draw grid, axes, and cube...
                \foreach \x in {-0.5,0,...,3}
                    \foreach \y in {-0.5,0,...,3} {
                        \draw[grid] (\x,-0.5) -- (\x,3);
                        \draw[grid] (-0.5,\y) -- (3,\y);
                    }
                \draw[axis] (0,0,0) -- (2.5,0,0) node[anchor=west]{$x$};
                \draw[axis] (0,0,0) -- (0,2.5,0) node[anchor=west]{$y$};
                \draw[axis] (0,0,0) -- (0,0,2.5) node[anchor=west]{$z$};

            % drow the walk
            \draw[cube] (1,1,0) -- (1.5,1,0) -- (1.5,1.5,1.5) -- (1,1.5,1.5) -- (1,1,0.5);
            \draw[cube] (1,1,0.5) -- (1.5,1,0.5) -- (1.5,1.5,2) -- (1,1.5,2) -- (1,1,1);
            \draw[cube] (1,1,1) -- (1.5,1,1) -- (1.5,1.5,2.5) -- (1,1.5,2.5) -- (1,1,1.5);
            \draw[cube] (1,1,1.5) -- (1.5,1,1.5) -- (1.5,1.5,3) -- (1,1.5,3) -- (1,1,2);
            \end{scope}
        \end{tikzpicture}
    };
    \node at (plot2.south) [above=120pt] {The walk $(\lambda(g);\tilde z)$.}; % Title for the first plot

    % Draw a curvy arrow between plots
    \draw[->, thick, red] ([xshift=0.2cm]plot1.east) to[out=20, in=160] ([xshift=-0.2cm]plot2.west);
    \end{tikzpicture}
To understand the proof it is important to understand the walk \( (\lambda(g); \tilde{z}) \) for some \( g \in \mathbb{H} \). The walk proceeds as follows:
\begin{align*}
\lambda(g)=&(2g_1,&&2g_2,   &&4g_3)&\stackrel{x}{\ra}\\
&(2g_1+1,&&2g_2,   &&4g_3)&\stackrel{y}{\ra}\\
&(2g_1+1,&&2g_2+1, &&4g_3+2g_1+1)&\stackrel{x^{-1}}{\ra}\\
&(2g_1,&&2g_2+1, &&4g_3+2g_1+1)&\stackrel{y^{-1}}{\ra}\\
\lambda(g)z=&(2g_1,&&2g_2,   &&4g_3+1)&\stackrel{x}{\ra}\\
& &&\vdots\\
&(2g_1+1,&&2g_2+1,   &&4g_3+2g_1+4)&\stackrel{x^{-1}}{\ra}\\
&(2g_1,&&2g_2+1,   &&4g_3+2g_1+4)&\stackrel{y^{-1}}{\ra}\\
\lambda(gz)=\lambda(g)z^4=&(2g_1,&&2g_2,   &&4g_3+4).
\end{align*}

Now we show that $\tilde f$ is a combinatorial embedding:

\textbf{\textit{"Roads do not intersect"}}:
Let \( e, t \in E \) be two non-adjacent edges. We want to show that \( \tilde{f}_e \cap \tilde{f}_t = \emptyset \). Suppose, for the sake of contradiction, that there exists a point \( \xi \in \tilde{f}_e \cap \tilde{f}_t \). Then, there are points \( p, q \in \mathbb{H} \) and generators \( s, u \in \{x, y, z\} \) such that \( \{p, ps\}\) is an edge in the walk \(f_e \) and \( \{q, qu\}\) is an edge in the walk  \(f_t \) such that $$\xi\in (\lambda(p);\tilde s)\cap (\lambda(q);\tilde u).$$

First, observe that \( \xi \) cannot be an endpoint of either walk. If \( \xi \) were an endpoint of one, it would also have to be the endpoint of the other, since the endpoints are the only points inside $2\Z\times 2\Z \times 4\Z$. However, \( \xi \) cannot be the endpoint of both walks because the endpoints \( \{\lambda(p), \lambda(ps)\} \) and \( \{\lambda(q), \lambda(qu)\} \) are distinct.

We analyze the cases based on the generator \( s \):
\begin{enumerate}

\item \textbf{If $\bm{s=x}$}: Since \( \xi \) is not an endpoint of $(\lambda(p);\tilde s)$, we have \( \xi = \lambda(p)x \).\\
\underline{If $u=x$}, then $\xi=\lambda(q)x$, which implies $p=q$, a contradiction. \\ 
\underline{If $u=y$}, then $\xi_1\equiv0\pmod 2$, a contradiction!\\
\underline{If $u=z$}, the condition \( \xi = \lambda(p)x \) implies \( \xi \in (2\mathbb{Z}+1) \times 2\mathbb{Z} \times 4\mathbb{Z} \), which holds in \((\lambda(q); \tilde{z}) \) only if \( \xi = \lambda(q)x \), a contradiction.
\item \textbf{If $\bm{s=y}$}: Then \( \xi = \lambda(p)y \). The cases where \( u \in \{x, y\} \) follow from similar reasoning as above. If \( u = z \), the condition \( \xi = \lambda(p)y \) implies \( \xi \in 2\mathbb{Z} \times (2\mathbb{Z}+1) \times 4\mathbb{Z} \), which is possible in \( (\lambda(q); \tilde{z}) \) only if \( \xi \in \{\lambda(q)z^2y, \lambda(qu)y\} \), a contradiction.
\item \textbf{If $\bm{s=z}$}.   If \( u \in \{x, y\} \), the cases were discussed previously. Assume \( u = z \).
 \begin{enumerate}
        \item  If \( (p_1, p_2) \neq (q_1, q_2) \), then \( (\xi_1, \xi_2) \) belongs to two distinct squares: \( [2p_1, 2p_1+1] \times [2p_2, 2p_2+1] \) and \( [2q_1, 2q_1+1] \times [2q_2, 2q_2+1] \), which is a contradiction.
        \item If $(p_1,p_2)=(q_1,q_2)$, recall that $\{p,pz\}$ and $\{q,qz\}$ are disjoint, hence $|p_3-q_3|\geq 2$.\\
            \underline{If $\xi_2\in 2\Z$}, we have $\xi_3\in [4p_3,4p_3+4]\cap [4q_3,4q_3+4]=\emptyset$, a contradiction.\\
            \underline{If $\xi_2\in 2\Z+1$}, we have $\xi_3\in [4p_3+2p_1+1,4p_3+2p_1+4]\cap [4q_3+2q_1+1,4q_3+2q_1+4]=\emptyset$, a contradiction.
    \end{enumerate}
    
\end{enumerate}
Thus, \( \xi \) cannot exist, and we have shown that "Roads do not intersect".

\smallskip

\textbf{\textit{"Vertices are isolated from roads"}}:
Let \( e \in E \) and \( v \in V \setminus e \). Suppose, for contradiction, that \( \tilde{f}v \in \tilde{f}_e \). Since \( \tilde{f}v = \lambda(fv) \), we know that \( \tilde{f}v \in (2\mathbb{Z})^2 \times 4\mathbb{Z} \). From the definition of \( \tilde{f}_e \), the only points of \( \tilde{f}_e \) that satisfy this condition are those in \( \lambda(f_e) \). Thus, \( \lambda(fv) \in \lambda(f_e) \), which implies that \( fv \in f_e \). This is a contradiction, since \( f \) is a combinatorial embedding.

\smallskip

Lastly, we show that \( \tilde{f} \) has a small diameter and volume:  
Notice that \( \operatorname{im} \tilde{f} \subseteq N_8(\lambda_2(\operatorname{im} f)) \) with respect to the $\{x,y\}$-metric. Therefore, there exists an absolute constant \( C \) such that:
$$diam(\tilde{f})\leq C\cdot diam(f)\quad\text{and}\quad vol(\tilde{f})\leq C\cdot vol(f).$$
\end{proof}

\subsection{Proposition \ref{proposition: general generating set} - General generating set}
The following proposition transforms combinatorial embeddings into $H$ into combinatorial embeddings into $H_S$ for a general finite generating set $S$. This lemma specifically requires starting with embeddings into $H$, as the generators $s = x$ and $s = y$ have the useful property that $\lambda_k(s) = s^k$.

\begin{proposition*}[General Generating Set]
For any generating set $S \subseteq \mathbb{H}$, there exists a constant $C=C(S)$ such that if $f:G \ra  H$ is a combinatorial embedding, then there is a combinatorial embedding $\tilde f:G\ra H_S$ such that
$$diam(\tilde{f})\leq C\cdot diam(f)\quad\text{and}\quad vol(\tilde{f})\leq C\cdot vol(f).$$
\end{proposition*}

\begin{proof}
Let $M$ be a constant such that $\frac{1}{M} \lVert p \rVert_S \leq \lVert p \rVert \leq M \lVert p \rVert_S$. Define
$$ m := \max(M \lVert x \rVert_S, M \lVert y \rVert_S) + 1. $$
Let $\lambda := \lambda_m$ and write $G=(V,E)$.

Now, define the combinatorial wiring $g: G \to H$ as follows: set $g|_V := \lambda \circ f|_V$, and for each $e \in E$, write $f_e = (p; a_1, \dots, a_r)$ where $a_i \in \{x, y, x^{-1}, y^{-1}\}$, and define 
$$ g_e := (\lambda(p); a_1^m, \dots, a_r^m). $$
This is a walk from $\lambda(p)$ to $\lambda(pa_1...a_r)$ - which confirms that $g$ is indeed a combinatorial  wiring.

\smallskip
Now, define the combinatorial wiring $\tilde{f}: G \to H_S$ to be the function $g$, but with a modification: whenever $g$ uses an edge labeled $x$, replace it with the walk $w_x \in S^\ast$ (the shortest word composed of elements from $S$ that represent $x$). Similarly, whenever $g$ uses an edge labeled $y$, replace it with the walk $w_y\in S^\ast$ (defined the same way as $w_x$).

The idea behind the proof is to show that the embedding $g$ has "large thickness," meaning that there is a significant distance between the roads of non-adjacent edges and a substantial distance between the image of a vertex and the road of any edge that does not contain it.
This extra space allows $\tilde{f}$ to replace each instance of $x$ with $w_x$ and each instance of $y$ with $w_y$ while keeping the roads distinct.

\smallskip

\subsubsection{$g$ has "large thickness"}
We prove here that:
\begin{itemize}
\item Roads are far from each other - for non-adjacent edges $e,e'\in E$ we have $d(g_e,g_{e'})\geq m$.
\item Vertices are far from roads - for $e\in E$ and $v\in V\setminus e$ we have $d(gv,g_e)\geq m$.
\end{itemize}
\underline{Roads are far from each other:}\\
Let $e, e' \in E$ be two non-adjacent edges. Take points $\xi \in g_e$ and $\xi' \in g_{e'}$, and show that $d(\xi, \xi') \geq m$.

By the definition of $g$, we have neighbors $p \sim p'$ and $q \sim q'$ in $H$, such that $\xi \in (\lambda(p); a^m)$ where $a := p^{-1} p'$, and $\xi' \in (\lambda(q); b^m)$ where $b := q^{-1} q'$. Thus, we have the equation:
\begin{equation} 
    \label{ineq_1}
    d(\lambda(p), \xi) + d(\xi, \lambda(p')) = m = d(\lambda(q), \xi') + d(\xi', \lambda(q')).
\end{equation}

Now, notice that $d(p, p') = d(q, q') = 1$. If either $d(p, q')$ or $d(p', q)$ equals 1, then both $d(p, q)$ and $d(p', q')$ must be $\geq 2$, since there are no cycles of length 3 in $H$. Therefore, either $d(p, q), d(p', q') \geq 2$, or $d(p, q'), d(p', q) \geq 2$. Without loss of generality, we assume that $d(p, q),d(p', q') \geq 2$.

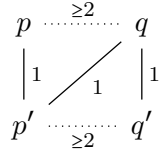
\[ \begin{tikzcd}
p \arrow[dash, dotted]{r}{\geq 2}
\arrow[dash]{d}{1} & q \arrow[dash]{d}{1} \arrow[dash]{dl}{1} \\%
p' \arrow[dash, dotted, swap]{r}{\geq 2}& q'
\end{tikzcd}
\]

Thus, $d(\lambda(p), \lambda(q))\geq 2m$ and $d(\lambda(p'), \lambda(q')) \geq 2m$. From equation \eqref{ineq_1}, we see that either $d(\lambda(p), \xi) + d(\lambda(q), \xi') \leq m$, or $d(\lambda(p'), \xi) + d(\lambda(q'), \xi') \leq m$.

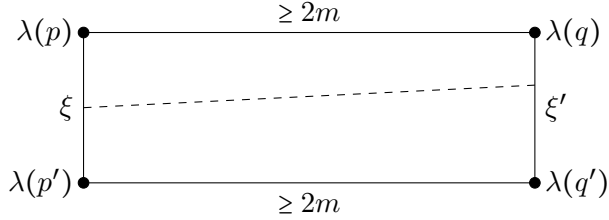
\begin{center}
\begin{tikzpicture}

  % Define coordinates
  \coordinate (A) at (0, 1);
  \coordinate (B) at (6, 1);
  \coordinate (C) at (0, -1);
  \coordinate (D) at (6, -1);
  \coordinate (E) at (0, 0);
  \coordinate (F) at (6, 0.3);

  % Draw horizontal lines
  \draw (A) -- (B) node[midway, above] {$\geq 2m$};
  \draw (C) -- (D) node[midway, below] {$\geq 2m$};

  % Draw dashed line
  \draw[dashed] (E) -- (F);

  % Draw vertical lines
  \draw (A) -- (C) node[midway, left] {$\xi$};
  \draw (B) -- (D) node[midway, right] {$\xi'$};

  % Draw vertical markers
  \filldraw (A) circle (2pt);
  \filldraw (C) circle (2pt);
  \filldraw (B) circle (2pt);
  \filldraw (D) circle (2pt);

  % Labels for the lambda points
  \node[anchor=east] at (A) {$\lambda(p)$};
  \node[anchor=east] at (C) {$\lambda(p')$};
  \node[anchor=west] at (B) {$\lambda(q)$};
  \node[anchor=west] at (D) {$\lambda(q')$};

\end{tikzpicture}
\end{center}

In conclusion, we must have $d(\xi, \xi') \geq m$.\\
\underline{Vertices are far from roads:}\\
Let $e \in E$ and $v \in V \setminus e$. Let $\xi \in g_e$, and we will show that $d(gv, \xi) \geq m$. By the definition of $g$, there are neighbors $p \sim p'$ in $f_e$ such that $\xi$ is at most $m$ away from both $\lambda(p)$ and $\lambda(p')$. Since $f$ is a combinatorial embedding the points $p,p'$ and $fv$ are distinct, and since $H$ does not contain cycles of length $3$ one of $d(fv, p)$ or $d(fv, p')$ must be at least 2. Without loss of generality, assume $d(fv, p) \geq 2$. This implies that 
$$d(\lambda(fv), \lambda(p)) \geq 2m.$$ 
Since we also know that $d(\xi, \lambda(p)) \leq m$, it follows that 
$$d(gv, \xi) = d(\lambda(fv), \xi) \geq m,$$ 
which is what we wanted to prove.

\subsubsection{$\tilde f$ is a combinatorial embedding}
$\tilde f|_V$ is injective since $\lambda $ and $f|_V$ are. So we only need to show that:

\textbf{"Roads do not intersect" }: Let $e, e' \in E$ be two non-adjacent edges. Every vertex in $\tilde{f}_e$ is at an $S$-distance of at most 
$$ \max\left(\frac{1}{2}\text{len}(w_x), \frac{1}{2}\text{len}(w_y)\right) = \max\left(\frac{1}{2}\lVert x \rVert_S, \frac{1}{2}\lVert y \rVert_S\right) < \frac{m}{2M} $$ 
from a point in $g_e$. Similarly, every point in $\tilde{f}_{e'}$ is at an $S$-distance of less than $\frac{m}{2M}$ from $g_{e'}$. Since the $\{x,y\}$-distance between $g_e$ and $g_{e'}$ is at least $m$, the $S$-distance between them is at least $\frac{m}{M}$. This ensures that there is no intersection between $\tilde f_e$ and $\tilde f_{e'}$.

\textbf{"Vertices are isolated from roads" }: Let $e \in E$ and $v \in V \setminus e$. As shown earlier, the $\{x,y\}$-distance between $gv$ and $g_e$ is at least $m$, which means the $S$-distance between $gv$ and $g_e$ is at least $\frac{m}{M}$.

Furthermore, we have already established that every point in $\tilde{f}_e$ is at an $S$-distance of less than $\frac{m}{2M}$ from $g_e$. Therefore, $\tilde{f}v = gv \notin \tilde{f}_e$.

\subsubsection{$\tilde f$ has small diameter and volume}
It remains to bound the diameter and volume of $\tilde{f}$:
Observe that $\text{Im}(\tilde{f}) \subseteq N_{\frac{m}{2M}}(\text{Im}(g))$ with respect to the $S$-distance, and $\text{Im}(g) \subseteq N_{\frac{m}{2}}(\lambda_m(\text{Im}(f)))$ with respect to the $\{x, y\}$-distance. Therefore, $\text{Im}(\tilde{f}) \subseteq N_{\frac{m}{2M}+\frac{mM}{2}}(\lambda_m(\text{Im}(f)))$ with respect to the $S$-distance.
Hence, there exists a constant $C = C(S)$ such that 
$$diam(\tilde{f})\leq C\cdot diam(f)\quad\text{and}\quad vol(\tilde{f})\leq C\cdot vol(f).$$
\end{proof}

\appendix

\section{Equivalence of asymptotic notations}\label{appendix: Equivalence of asymptotic notations}
We prove here the equivalence of several asymptotic inequality notations, which is necessary in some parts of the paper.

Let \( f, g: \mathbb{N} \to \N\cup\{\infty \} \). 

\begin{itemize}
\item Denote \( f \preceq g \) if \[
\exists C\forall n:f(n) \leq Cg(n).
\]
\item Recall that $f\lesssim g$ means \[
\exists C\forall n :f(n) \leq Cg(Cn)+C.
\] This is the definition of the symbol $\lesssim $ in this paper and also in \cite{BH}.
\item Denote $f\lessapprox g$ if \[
\exists C\forall n :f(n) \leq Cg(Cn+C)+C.
\] This is the definition of the symbol $\lesssim $ in \cite{HMT}.

\end{itemize}

\begin{lemma}\label{lemma: asymtotical notations smaller than n^r}
Let $r>0$. The following are equivalent:
\begin{enumerate}
\item $f\preceq n^r$.
    \item $f\lesssim n^r$.
    \item $f\lessapprox n^r$.
\end{enumerate}
\end{lemma}
\begin{proof}
$(1)\Rightarrow (2) \Rightarrow (3)$ it trivial. We prove $(3)\Rightarrow (1)$: There is $C\geq 1$ such that for every $n\in \N$: $$f(n)\leq C(Cn+C)^r+C\leq (2C)^{r+1}n^r$$. Hence $f\preceq n^r$.
\end{proof}

\begin{lemma}\label{lemma: asymtotical notations grater than n^r}
Let $r>0$. If $f$ is monotone increasing, then the following are equivalent: 
\begin{enumerate}
\item $n^r\preceq f$.
    \item $n^r\lesssim f$.
    \item $n^r \lessapprox f$.
\end{enumerate}
\end{lemma}
\begin{proof}
$(1)\Rightarrow (2) \Rightarrow (3)$ follows immediately from monotonicity of $f$. We prove $(3)\Rightarrow (1)$: 
For $n< (4C)^{1+\frac{1}{r}}$ we have $n^r< (4C)^{r+1}\leq (4C)^{r+1}f(n)$. And for \( n \geq (4C)^{1+\frac{1}{r}} \),
define $n':=\lfloor \frac{n-C}{C}\rfloor$, and notice that$$n'\geq \frac{n}{C}-2\geq \frac{n}{2C}\geq (2C)^\frac{1}{r},$$ and \[
f(n) \geq f(Cn'+C) \geq \frac{n'^r - C}{C} = \frac{1}{C}n'^r - 1 \geq \frac{1}{2C}n'^r \geq \frac{1}{(2C)^{r+1}}n^r.
\]
Thus, \( n^r \leq (2C)^{r+1} f(n)\leq (4C)^{r+1} f(n) \). Therefore, \( n^r \preceq  f(n) \).
\end{proof}

Now we justify one implication from the paper:
\begin{lemma}\label{lemma: 1.17 imply fact}
\cite[theorem 1.17]{BH} implies the statement:

Let \( d\geq 3 \), and let \( Y \) be a graph with separation profile \( \simeq n^\alpha \) for $\alpha\in (0,1]$. There exists a constant \( c > 0 \) such that for every \( n \in \mathbb{N} \), there exists a graph with \( n \) vertices and maximum degree \( \leq d \), for which any $1$-thick embedding into \( Y \) has volume \( \geq cn^{1/\alpha} \).
\end{lemma}
\begin{proof}
In the context of \cite[theorem 1.17]{BH}, take \( X \) to be the graph formed by the union of all finite graphs with maximum degree \( \leq d \), since there is family of  $3$-regular expander graphs we infer that $sep_X\gtrsim n$. So, by \cite[theorem 1.17]{BH} we get that $wir^k_{X\ra Y}\gtrsim n^{1/\alpha}$ for any $k\in \N$. Where $$wir^k_{X\ra Y}(n):=\max_{\G\leq X,|\G|\leq n}\ \inf_{f:
\G\ra Y \text{ is a coarse }k-wiring} vol(f).$$ Since $wir^k_{X\ra Y}$ is monotone increasing we use lemma \ref{lemma: asymtotical notations grater than n^r} and get $wir^k_{X\ra Y}\succeq n^{1/\alpha}$. Which means $\exists c\forall n: wir^k_{X\ra Y}\geq cn^{\frac{1}{\alpha}}$. Since every $1$-thick embedding is in particular a coarse $k$-wiring we finish.
\end{proof}

\section{Convert embedding into $\R^d$ to embedding into $\Z^d$}
Let $d\geq 3$. Let \( S \subseteq \mathbb{Z}^d \) be a finite generating set. Suppose we have a $1$-thick embedding \( f: G \to \mathbb{R}^d \) for a finite graph \( G \). Our goal is to construct a $1$-thick embedding \( \tilde{f}: G \to \mathrm{Cay}(\mathbb{Z}^d, S) \) with same volume as \( f \), up to a constant (independent of \( G \)).

Define \( Z_d := \mathrm{Cay}(\mathbb{Z}^d, \{e_1, \ldots, e_d\}) \). It is sufficient to show the existence of a $1$-thick embedding \( \tilde{f}: G \to Z_d \) with the same volume (up to a constant) as \( f \). Once we have this embedding, we can scale \( \tilde{f} \) by a large enough factor (depending only on \( S \)) and then replace each edge labeled by \( e_i \) with the shortest \( S \)-word that represents it. This will still result in a $1$-thick embedding.

Thus, we focus on constructing a $1$-thick embedding into \( Z_d \), which can also be viewed as a subset of \( \mathbb{R}^d \).

The idea is to rescale \( f \) by factor of \( 3\sqrt{d} \), which gives a $3\sqrt{d}$-thick embedding, and then show that every curve \( \gamma: [0,1] \to \mathbb{R}^d \) can be deformed into a curve \( \hat{\gamma}: [0,1] \to Z_d \subset \mathbb{R}^d \), with endpoints very close to the original ones such that \( \mathrm{Im}(\hat{\gamma}) \subseteq N_{\sqrt{d}}(\mathrm{Im}(\gamma)) \). This completes the argument since this deformation converts a $3\sqrt d$-thick embedding in \( \mathbb{R}^d \) into a $1$-thick embedding in \( Z_d \), as we will see.

Define \( \lfloor \cdot \rfloor: \mathbb{R}^d \to \mathbb{Z}^d \) as the pointwise floor function. We will prove the following:

\begin{lemma}
Let \( \gamma: [0,1] \to \mathbb{R}^d \) be a curve. There exists a curve \( \hat{\gamma}: [0,1] \to Z_d \subset \mathbb{R}^d \) such that \( \hat{\gamma}(0) = \lfloor \gamma(0) \rfloor \), \( \hat{\gamma}(1) = \lfloor \gamma(1) \rfloor \), and \( \mathrm{Im}(\hat{\gamma}) \subseteq N_{\sqrt{d}}(\mathrm{Im}(\gamma)) \) with respect to the \( \|\cdot\|_2 \) metric.
\end{lemma}

After proving this lemma, we will complete the construction by defining \( \tilde{f}: G \to Z_d \) as follows: Let \( g := 3\sqrt{d} f \). For each vertex \( v \in VG \), define \( \tilde{f}(v) := \lfloor g(v)\rfloor \). For each edge \( e \subseteq G \), since \( g|_e \) is a curve in \( \mathbb{R}^d \), extend $\tilde f$ to $e$ by \( \tilde{f}|_e := \hat{g|_e} \). This definition is well-defined because, for any vertex \( v \in e \in EG \), both \( \hat{g|_e} \) and \( \tilde{f} \) agree at \( v \), as they both return \( \lfloor g(v) \rfloor \). 

The function $\tilde f$ is a $1$-thick embedding since: Recall that the metric on \( Z_d \) is induced by its graph structure, which is equivalent to the \( \|\cdot\|_1 \) metric as a subset of \( \mathbb{R}^d \). Now, let \( e, e' \in EG \) be non-adjacent edges. We have
\[
d_2(\tilde{f}(e), \tilde{f}(e')) \geq d_2(g(e), g(e')) - 2\sqrt{d} \geq \sqrt{d} \geq 1,
\]
where \( d_i \) is the distance induced by the \( \|\cdot\|_i \) metric for \( i = 1, 2 \). Since \( \|\cdot\|_1 \geq \|\cdot\|_2 \), we have \( d_1(\tilde{f}(e), \tilde{f}(e')) \geq 1 \), as required. A similar argument applies to the case of an edge and a vertex not incident to it.

\begin{proof}[proof of the lemma]
For \( q \in \mathbb{Z}^d \), define the cube \( C_q := q + [0,1]^d \). Let \( Q := \{ q \in \mathbb{Z}^d \mid C_q \cap \mathrm{Im}(\gamma) \neq \emptyset \} \) be the set of points in $\Z^d$ that "touch" \( \gamma \). Define a graph \( \Gamma \) with vertex set \( Q \), where there is an edge between \( q \) and \( q' \) if and only if \( C_q \cap C_{q'} \neq \emptyset \), or equivalently, the cubes \( C_q \) and \( C_{q'} \) share a vertex. Let \( x := \lfloor \gamma(0) \rfloor \) and \( y := \lfloor \gamma(1) \rfloor \), noting that \( x, y \in Q \). We will show that \( d_\Gamma(x, y) < \infty \), which conclude the proof: In this case, there exists a finite sequence of cubes \( C_x = C_{q_1}, C_{q_2}, \dots, C_{q_r} = C_y \), where $q_i\in Q$ and each consecutive pair \( C_{i+1},C_i \) shares a vertex. This gives a path from \( x \), a vertex of \( C_x \), to \( y \), a vertex of \( C_y \), passing through the union of the $1$-skeletons of the cubes \( C_{q_i} \). This path lies in \( Z_d \), and it is the desired curve \( \hat{\gamma} \), as each \( C_i \) is within \( N_{\sqrt{d}}(\mathrm{Im}(\gamma)) \) with respect to the \( \|\cdot\|_2 \) metric.

To prove that \( d_\Gamma(x, y) < \infty \), we will prove that \( d_\Gamma(x, y)\leq |Q| \). Define 
\[
A := \{ t \in [0,1] \mid \exists q \in Q : \gamma(t) \in C_q, d_\Gamma(x, q) \leq |Q| - 1 \}.
\]
Let \( b := \sup A \). Since \( 0 \in A \), the set \( A \) is non-empty. Therefore, there exists a sequence \( a_k \to b \) with \( a_k \in A \). For each \( a_k \), there is a \( q(a_k) \in Q \) such that \( \gamma(a_k) \in C_{q(a_k)} \) and \( d(x, q(a_k)) \leq |Q| - 1 \). Since \( Q \) is finite, there exists a subsequence \( a_{k_j} \to b \) such that \( q(a_{k_j}) \) is constant, say \( q \). Therefore, \( \gamma(a_{k_j}) \in C_q \) for all \( j \), and since \( C_q \) is closed, we have \( \gamma(b) \in C_q \).

If \( b = 1 \), then \( \gamma(b) \in C_y \). In this case, the cubes \( C_q \) and \( C_y \) share a point (the point \( \gamma(b) \)), so \( d_\Gamma(q, y) \leq 1 \), which implies \( d_\Gamma(x, y) \leq |Q| \), as required.

If \( b < 1 \), let \( t_k \to b \) be a sequence with \( t_k \in (b, 1] \), meaning \( t_k \notin A \). Thus, there exists \( q(t_k) \in Q \) such that \( \gamma(t_k) \in C_{q(t_k)} \) and \( d_\Gamma(x, q(t_k)) \geq |Q| \). As before, we can narrow down to a subsequence, or just assume that \( q(t_k) \) is constant, say \( q' \). Therefore, \( \gamma(t_k) \in C_{q'} \), which implies \( \gamma(b) \in C_{q'} \).

Since \( C_q \cap C_{q'} \supseteq \{\gamma(b)\} \), we have \( d_\Gamma(q, q') \leq 1 \). Given that \( d_\Gamma(x, q) < |Q| \) and \( d_\Gamma(x, q') \geq |Q| \), it follows that \( d_\Gamma(x, q') = |Q| \), which is a contradiction! The distance in a graph cannot be equal to the number of vertices; it must either be infinite or less than the number of vertices by at least 1.
\end{proof}
\subsection*{Acknowledgments}
We are thankful to Itai Benjamini and Elad Tzalik for guidance and helpful discussions.

\end{document}